\documentclass[12pt]{amsart}
\usepackage{amsfonts, amssymb, latexsym, graphicx, hyperref, amsmath}
\usepackage[vcentermath]{youngtab}
\usepackage{tikz}
\usepackage{tikz-cd}
\usetikzlibrary{calc, shapes, backgrounds,arrows,positioning,decorations.pathmorphing}

\newtheorem{theorem}{Theorem}[section]
\newtheorem{proposition}[theorem]{Proposition}

\newtheorem{lemma}[theorem]{Lemma}

\newtheorem{claim}[theorem]{Claim}

\newtheorem*{claim*}{Claim}

\newtheorem{Main Conjecture}[theorem]{Main Conjecture}
\newtheorem{conjecture}[theorem]{Conjecture}

\newtheorem{definition}[theorem]{Definition}
\theoremstyle{remark}
\newtheorem{example}[theorem]{Example}
\newtheorem{Remark}[theorem]{Remark}

\theoremstyle{plain}

\newcommand{\cellsize}{11}
\newlength{\cellsz} 
\newsavebox{\cell}
\sbox{\cell}{\begin{picture}(\cellsize,\cellsize)
\put(0,0){\line(1,0){\cellsize}}
\put(0,0){\line(0,1){\cellsize}}
\put(\cellsize,0){\line(0,1){\cellsize}}
\put(0,\cellsize){\line(1,0){\cellsize}}
\end{picture}}
\newcommand\cellify[1]{\def\thearg{#1}\def\nothing{}%
\ifx\thearg\nothing
\vrule width0pt height\cellsz depth0pt\else
\hbox to 0pt{\usebox{\cell} \hss}\fi%
\vbox to \cellsz{
\vss
\hbox to \cellsz{\hss$#1$\hss}
\vss}}
\newcommand\tableau[1]{\vtop{\let\\\cr
\baselineskip -16000pt \lineskiplimit 16000pt \lineskip 0pt
\ialign{&\cellify{##}\cr#1\crcr}}}

\newcommand{\excise}[1]{}

\title{Newell-Littlewood numbers II:\\ extended Horn inequalities}
\author{Shiliang Gao}
\author{Gidon Orelowitz}
\author{Alexander Yong}
\address{Dept.~of Mathematics, University of Illinois at Urbana-Champaign, Urbana, IL 61801}

\email{sgao23@illinois.edu, gidono2@illinois.edu, ayong@illinois.edu}

\date{December 3, 2021}
\begin{document}
\pagestyle{plain}

\maketitle
\begin{abstract}
The Newell-Littlewood numbers $N_{\mu,\nu,\lambda}$ are
tensor product multiplicities of Weyl modules for classical Lie groups, in the stable limit. For which triples of partitions $(\mu,\nu,\lambda)$ does $N_{\mu,\nu,\lambda}>0$ hold? The Littlewood-Richardson coefficient case  is solved by the Horn inequalities (in work of
A.~Klyachko and A.~Knutson-T.~Tao). We extend these celebrated linear inequalities
to a much larger family, suggesting a general solution.
\end{abstract}

\begin{center}
\emph{In honor of Ian Goulden and David Jackson, and their groundbreaking discoveries}
\end{center}

\section{Introduction}

\subsection{Background}
This is a sequel to \cite{GOY}. We study \emph{Newell-Littlewood numbers}
\cite{Newell, Littlewood} 
\begin{equation}
\label{eqn:Newell-Littlewood}
N_{\mu,\nu,\lambda}=\sum_{\alpha,\beta,\gamma} c_{\alpha,\beta}^{\mu}
c_{\alpha,\gamma}^{\nu}c_{\beta,\gamma}^{\lambda};
\end{equation}
the indices are partitions 
in 
${\sf Par}_n=\{(\lambda_1, \lambda_2, \ldots, \lambda_n)\in {\mathbb Z}_{\geq 0}^n:
\lambda_1\geq \lambda_2\geq \ldots \geq \lambda_n\}$.
Also, $c^{\mu}_{\alpha,\beta}$ is the  
\emph{Littlewood-Richardson coefficient}. The Newell-Littlewood
numbers are tensor product multiplicities for the irreducible representations of a classical Lie algebra ${\mathfrak g}$ in the ``stable limit''; we refer the reader to \cite{GOY} for  additional background and references, such as \cite{Hahn}.

Consider the problem:
\begin{center}
\emph{Classify  $(\mu,\nu,\lambda)\in {\sf Par}_n^3$ such that 
$N_{\mu,\nu,\lambda}>0$.}
\end{center}  
Since   $N_{\mu,\nu,\lambda}\!\!=\!\!c_{\mu,\nu}^{\lambda}$ if $|\lambda| \!=\! |\mu|\!+\!|\nu|$ \cite[Lemma~2.2(II)]{GOY},
a subproblem asks when $c_{\mu,\nu}^{\lambda}>0$? The solution to that case is 1990's combined breakthrough work of
A.~Klyachko \cite{Klyachko} and A.~Knutson-T.~Tao \cite{KT}. For $I=\{i_1< \cdots<i_d\}\subseteq {\mathbb Z}_{>0}$,
let 
\[\tau(I):=(i_d-d\geq \cdots\geq i_2-2 \geq i_1-1)\in {\sf Par}_d.\]

\begin{theorem}\label{thm:classicalHorn}(\cite{Klyachko}, \cite{KT})
Let $\mu,\nu,\lambda \in {\sf Par}_n$ such that
$|\lambda|=|\mu|+|\nu|$. Then
$c_{\mu,\nu}^{\lambda}>0$ if and only if
for every $d<n$, and every triple of subsets $I,J,K\subseteq [n]$ of cardinality $d$ such that $c_{\tau(I),
\tau(J)}^{\tau(K)}> 0$,
\begin{equation}\label{eq:ineq}
\sum_{k\in K}\lambda_k\leq
\sum_{i\in
I}\mu_i+\sum_{j\in J}\nu_j.\end{equation}
\end{theorem}

The recursive  \emph{Horn inequalities} (\ref{eq:ineq}) were introduced in A.~Horn's 1962 paper \cite{Horn}. The inequalities
have a pre-history \cite{Fulton,Bhatia}.

Let ${\mathfrak g}$ be a semisimple complex Lie algebra, $\Lambda_+$ be the set of dominant integral weights, and $L_{\mathfrak g}$ be the root lattice.
Suppose $V_{\lambda}$ is the irreducible
representation of ${\mathfrak g}$ indexed by $\lambda\in \Lambda_+$. Define multiplicities $m_{\mu,\nu}^{\lambda}$ by
\[V_\mu\otimes V_{\nu}=\bigoplus_{\lambda\in \Lambda_+} V_{\lambda}^{\oplus m_{\mu,\nu}^{\lambda}}.\]
The \emph{tensor semigroup} is
\[{\sf Tensor}({\mathfrak g})=\{(\mu,\nu,\lambda)\in \Lambda_+^3: m_{\mu,\nu}^{\lambda}>0\}.\]
Compare this with the \emph{saturated tensor semigroup},
\[{\sf SatTensor}({\mathfrak g})=\{(\mu,\nu,\lambda)\in \Lambda_+^3: \mu+\nu-w_0\cdot\lambda \in L_{\mathfrak g} \text{ and } \exists t\in {\mathbb Z}_{>0}, m_{t\mu,t\nu}^{t\lambda}>0\}\]
where $w_0$ is the longest length element of the Weyl group associated to ${\mathfrak g}$.

There are generalized Horn inequalities describing ${\sf SatTensor}({\mathfrak g})$ \cite{BK}. 
Since $N_{\mu,\nu,\lambda}$ is a tensor product multiplicity for ${\mathfrak g}$ of classical type $B,C,D$, 
these results
are related to our classification problem, but do not solve it. Classifying $N_{\mu,\nu,\lambda}>0$ concerns  ${\sf Tensor}({\mathfrak g})$ rather than the possibly different
${\sf SatTensor}({\mathfrak g})$. In type $A$, the \emph{saturation theorem} \cite{KT} implies \[{\sf Tensor}({\mathfrak s}{\mathfrak l}(n))
={\sf SatTensor}({\mathfrak s}{\mathfrak l}(n)).\] 
For the other classical types, saturation is either false, or not known (see \cite{Zel, KM}).\footnote{Since this paper was submitted, N.~Ressayre and the
authors \cite{NLIII} further study the relationship of our classification problem to \cite{BK}.}

N.~Ressayre \cite{Ress} introduces different generalized Horn inequalities that hold when the 
\emph{Kronecker coefficient} $g_{\mu,\nu,\lambda}$  is nonzero. Those coefficients are also tensor product multiplicities, but for Specht modules, not Weyl modules. 

\subsection{Main results}
We suggest an answer to our problem, by introducing a large, new family of inequalities extending (\ref{eq:ineq}). 

\begin{definition}
\label{def:main}
An extended Horn inequality is
\begin{equation}
\label{Ineq:Grand}
0\leq \sum_{i\in A} \mu_i - \sum_{i\in A'} \mu_i + \sum_{j\in B} \nu_j - \sum_{j\in B'} \nu_j + \sum_{k\in C} \lambda_k - \sum_{k\in C'} \lambda_k 
\end{equation}  
\end{definition}
\noindent
where $A,A',B,B',C,C' \subseteq [n]:=\{1,2,\ldots,n\}$ satisfy
\begin{itemize}
\item[(I)]
$A \cap A'= B \cap B' = C \cap C' = \emptyset$
\item[(II)]
$|A| = |B'|+|C'|,|B| =|A'|+|C'|, |C| = |A'|+|B'|$
\item[(III)]
There exist $A_1,A_2,B_1,B_2,C_1,C_2\subseteq [n]$ such that:
\begin{enumerate}
\item
$|A_1|=|A_2|=|A'|, |B_1|=|B_2|=|B'|, |C_1|=|C_2|=|C'|$
\item
$c^{\tau(A')}_{\tau(A_1),\tau(A_2)}, c^{\tau(B')}_{\tau(B_1),\tau(B_2)}, c^{\tau(C')}_{\tau(C_1),\tau(C_2)}>0$
\item
$c^{\tau(A)}_{\tau(B_1),\tau(C_2)}, c^{\tau(B)}_{\tau(C_1),\tau(A_2)}, c^{\tau(C)}_{\tau(A_1),\tau(B_2)}>0.$
\end{enumerate}
\end{itemize}

This family contains a number of simpler-to-state subfamilies, including the Horn inequalities (\ref{eq:ineq}) and those considered in \cite{GOY}; see Proposition~\ref{prop:specificineq}. This is our main result:
\begin{theorem}
\label{thm:intro}
$(\mu,\nu,\lambda)\in {\sf Par}_n^3$
satisfies (\ref{Ineq:Grand}) if $N_{\mu,\nu,\lambda}>0$.
\end{theorem}

We prove Theorem~\ref{thm:intro} in Section~\ref{sec:2}.  Another necessary condition for 
$N_{\mu,\nu,\lambda}>0$ is a parity requirement \cite[Lemma~2.2]{GOY}:
\begin{equation}
\label{eqn:parity}
|\mu|+|\nu|+|\lambda| \equiv 0  \ (\!\!\!\!\! \mod 2).
\end{equation}

Let ${\mathcal G}_n$ be the tuples $(A,A',B,B',C,C')$ satisfying (I)--(III).
We believe that (\ref{Ineq:Grand}) combined with (\ref{eqn:parity}) provides a classification.
  
\begin{conjecture}
\label{thm:Classification}
If $(\mu,\nu,\lambda)\in {\sf Par}_n^3$ satisfies (\ref{eqn:parity}), and (\ref{Ineq:Grand}) holds for every $(A,A',B,B',C,C')\in \mathcal{G}_n$, then
$N_{\mu,\nu,\lambda}>0$.\footnote{The ``saturated version'' of this conjecture is now
\cite[Theorem~1.5]{NLIII}.}
\end{conjecture}

This is exhaustively computer-checked, with D.~Brewster's assistance, for up to $n\leq 4$ and $|\mu|,|\nu|,|\lambda|\leq 20$, for $n=5$ and $|\mu|,|\nu|,|\lambda|\leq 16$, and for $n=6$ and $|\mu|,|\nu|,|\lambda|\leq 12$. 

Since the extended Horn inequalities are homogeneous in $\mu_i,\nu_j,\lambda_k$, Conjecture~\ref{thm:Classification}
immediately implies the \emph{Newell-Littlewood saturation conjecture} \cite[Conjecture~5.4]{GOY}:
\begin{equation}
\label{eqn:NLsat}
\text{If $(\mu,\nu,\lambda)\in {\sf Par}_n^3$ and (\ref{eqn:parity}) holds, then $N_{t\mu,t\nu,t\lambda}>0\Rightarrow
N_{\mu,\nu,\lambda}>0$,}
\end{equation} 
where $t\in {\mathbb Z}_{>0}$ and $t\mu=(t\mu_1,t\mu_2,\ldots)\in {\sf Par}_n$. However, unlike
the situation of \cite{KT}, we have no proof that (\ref{eqn:NLsat}) $\Rightarrow$ Conjecture~\ref{thm:Classification}.\footnote{The proof of this direction is now \cite[Corollary~10.5]{NLIII}.}

The Newell-Littlewood numbers enjoy a symmetry (\cite[Lemma~2.2(I)]{GOY}), namely, that 
\begin{equation}
\label{eqn:S3symmetry}
N_{\mu^{(1)},\mu^{(2)},\mu^{(3)}}=
N_{\mu^{(\sigma(1))},\mu^{(\sigma(2))},\mu^{(\sigma(3))}}, \text{\ for any $\sigma\in {\mathfrak S}_3$}.
\end{equation}
By construction, the extended Horn inequalities respect this ${\mathfrak S}_3$-symmetry. It is also evident
from the definition that ${\mathcal G}_{n}\subset {\mathcal G}_{n+1}$.

In Section~\ref{sec:3} we prove the ``Pieri case''  of Conjecture~\ref{thm:Classification}.

\begin{theorem}\label{thm:rowcol}
Conjecture~\ref{thm:Classification} is true when at least one of $\mu,\nu,\lambda$ is a row or a column.
\end{theorem}

In contrast with \cite{BK, Ress}, our methods are completely combinatorial, starting from (\ref{eqn:Newell-Littlewood}). 
The main work was the uncovering of the form of the inequalities (\ref{Ineq:Grand}).

\section{Proof of Theorem~\ref{thm:intro}, subfamilies, and stability}\label{sec:2}

\subsection{Proof of Theorem~\ref{thm:intro}:} 
We need this result of E.~Briand-R.~Orellana-M.~Rosas:

\begin{theorem}\label{lemma:boxsymmetry}\cite[Theorem~4]{BOR}
	For any partition $\lambda,\mu$ and $\nu$ such that $\lambda \subseteq (n^{k+l}),\mu\subseteq (n^k)$ and $\nu\subseteq (n^l)$, 
\[c^{\lambda}_{\mu,\nu} = c^{\lambda^{\vee(n^{k+l})}}_{\mu^{\vee(n^k)},\nu^{\vee(n^l)}},\]
where if $\theta \subseteq (n^m)$, $\theta^{\vee(n^m)}$ is the partition obtained by taking the complement of $\theta\subseteq n\times m$ and rotating $180$-degrees.  
\end{theorem}

We use the following reformulation of the main definition.

\begin{lemma}\label{lemma:reformulate}
In Definition~\ref{def:main}, it is equivalent to replace 
Condition (III)(3) with
\begin{itemize} 
\item[(III)(3)'] $m_A := \min(|B'|, |C'|)$, $m_B := \min(|A'|, |C'|)$, $m_C := \min(|A'|, |B'|)$
\[0< c^{\tau(A^c\cup [n+1,n+|A|-m_A])}_{\tau(B_1^c\cup [n+1,n+|B_1|-m_A]),\tau(C_2^c\cup [n+1,n+|C_2|-m_A])},\]
\[0< c^{\tau(B^c\cup [n+1,n+|B|-m_B])}_{\tau(C_1^c\cup [n+1,n+|C_1|-m_B]),\tau(A_2^c\cup [n+1,n+|A_2|-m_B])},\]
\[0< c^{\tau(C^c\cup [n+1,n+|C|-m_C])}_{\tau(A_1^c\cup [n+1,n+|A_1|-m_C]),\tau(B_2^c\cup [n+1,n+|B_2|-m_C])}.\]
(Above $A^c,B^c,C^c\subseteq [n]$.)
\end{itemize}
\end{lemma}
\begin{proof}
Notice that 
\begin{align*}
\ & \tau(C^c\cup [n+1,n+|C|-m_C])\\
=\ & \tau(C^c) \cup (n+1-(n-|C|+1))^{|C|-m_C}\\
=\ & \tau(C^c) \cup |C|^{|C|-m_C}
\end{align*}
Now, $\tau(C^c)$ is in fact $\tau(C)^\vee$ \emph{and transposed}. Thus 
\[ \tau(C^c) \cup |C|^{|C|-m_C}=(\tau(C)^{\vee}+(|C|-m_C)^{|C|})',\]
where for a partition $\alpha$, we denote $\alpha'$ to be the transpose of $\alpha$.

A similar equality holds for the other two arguments. Hence
\begin{align*}
\ & c^{\tau(C^c\cup [n+1,n+|C|-m_C])}_{\tau(A_1^c\cup [n+1,n+|A_1|-m_C]),\tau(B_2^c\cup [n+1,n+|B_2|-m_C])}\\
= \ & c^{(\tau(C)^{\vee}+|C|^{|C|-m_C})'}_{(\tau(A_1)^{\vee}+|A_1|^{|A_1|-m_C})', (\tau(B_2)^{\vee}+|B_2|^{|B_2|-m_C})'}\\
= \ & c^{\tau(C)^{\vee}+(|C|-m_C)^{|C|}}_{\tau(A_1)^{\vee}+(|A_1|-m_C)^{|A_1|}, \tau(B_2)^{\vee}+(|B_2|-m_C)^{|B_2|}},
\end{align*}
where we have used the standard symmetry $c_{\alpha,\beta}^{\gamma}=c_{\alpha',\beta'}^{\gamma'}$. 

Since 
\[\tau(C)^\vee \subseteq (n-|C|)^{|C|}, \tau(A_1)^\vee\subseteq (n-|A_1|)^{|A_1|}\text{ and }\tau(B_2)^\vee \subseteq (n-|B_2|)^{|B_2|},\]
one has
\begin{align*}
\ &\tau(C)^{\vee}+(|C|-m_C)^{|C|}\subseteq (n-m_C)^{|C|},\\
\ &\tau(A_1)^{\vee}+(|A_1|-m_C)^{|A_1|}\subseteq (n-m_C)^{|A_1|},\\
\ &\tau(B_2)^{\vee}+(|B_2|-m_C)^{|B_2|}\subseteq (n-m_C)^{|B_2|}.
\end{align*}
Observe
\begin{align*}
\ &(\tau(C)^{\vee}+(|C|-m_C)^{|C|})^{\vee((n-m_C)^{|C|})} = \tau(C),\\
\ &(\tau(A_1)^{\vee}+(|A_1|-m_C)^{|A_1|})^{\vee((n-m_C)^{|A_1|})} = \tau(A_1),\\
\ &(\tau(B_2)^{\vee}+(|C|-m_C)^{|B_2|})^{\vee((n-m_C)^{|B_2|})} = \tau(B_2).
\end{align*}
Since, by Condition (II), $|C| = |A_1|+|B_2|$, we can apply Theorem~\ref{lemma:boxsymmetry} and obtain
\[c^{\tau(C)^{\vee}+(|C|-m_C)^{|C|}}_{\tau(A_1)^{\vee}+(|A_1|-m_C)^{|A_1|}, \tau(B_2)^{\vee}+(|B_2|-m_C)^{|B_2|}}=  c^{\tau(C)}_{\tau(A_1),\tau(B_2)}.\]
The other two cases are similarly proved to be equivalent with the corresponding condition in Definition~\ref{def:main}.
\end{proof}

Let $(A,A',B,B',C,C')\in\mathcal{G}_n$, and let $(A_1,A_2,B_1,B_2,C_1,C_2)$ be as in (III).   Let $\mu,\nu,\lambda\in {\sf Par}_n$ satisfy $N_{\mu,\nu,\lambda}>0$.  By (\ref{eqn:Newell-Littlewood}), there exist $\alpha, \beta, \gamma\in {\sf Par}_n$ such that 
\[c^\mu_{\alpha, \beta}, c^\nu_{\alpha, \gamma}, c^\lambda_{\beta, \gamma} >0.\]   
By Theorem~\ref{thm:classicalHorn}, $(\mu,\alpha, \beta)$, $(\nu,\alpha,\gamma)$, and $(\lambda,\beta, \gamma)$ satisfy  (\ref{eq:ineq}).  In particular, 
\begin{equation}
\sum_{i\in A'} \mu_i \leq \sum_{i\in A_1} \beta_i + \sum_{i\in A_2}\alpha_i, \sum_{j\in B'} \nu_j \leq \sum_{j\in B_1} \alpha_j + \sum_{j\in B_2}\gamma_j, \sum_{k\in C'} \lambda_k \leq \sum_{k\in C_1} \gamma_k + \sum_{k\in C_2}\beta_k
\end{equation}

In addition, since $\mu,\alpha,\beta\in {\sf Par}_n$, and in view of Lemma~\ref{lemma:reformulate},
\begin{align*}
\sum_{i\in A} \mu_i =&\  |\mu| - \sum_{i\in A^c} \mu_i\\
 = & \ |\mu| - \sum_{i\in A^c\cup [n+1,n+|A|-m_A]} \mu_i
 \end{align*}
 \begin{align*}
  \geq & \ |\mu| - \sum_{i\in B_1^c\cup [n+1,  n+|B_1|-m_A]} \alpha_i - \sum_{i\in C_2^c\cup [n+1,  n+|C_2|-m_A]} \beta_i\\
  = & \ |\mu| - \sum_{i\in B_1^c} \alpha_i - \sum_{i\in C_2^c} \beta_i\\
 = & \ |\alpha| - \sum_{i\in B_1^c} \alpha_i + |\beta| - \sum_{i\in C_2^c} \beta_i\\
  = & \ \sum_{i\in B_1} \alpha_i + \sum_{i\in C_2} \beta_i.
\end{align*}
By the same logic, \[\sum_{j\in B} \nu_j \geq \sum_{j\in C_1} \gamma_j + \sum_{j\in A_2} \alpha_j \text{ and } \sum_{k\in C} \gamma_k \geq \sum_{k\in A_1} \beta_k + \sum_{k\in B_2} \gamma_k.\]

Therefore, 
\begin{align*}
\sum_{i\in A'} \mu_i + \sum_{j\in B'} \nu_j + \sum_{k\in C'} \lambda_k \leq &\  (\sum_{i\in A_1} \beta_i + \sum_{i\in A_2}\alpha_i) + (\sum_{j\in B_1} \alpha_j + \sum_{j\in B_2}\gamma_j) + (\sum_{k\in C_1} \gamma_k + \sum_{k\in C_2}\beta_k)\\
= &\  (\sum_{i\in B_1} \alpha_i + \sum_{i\in C_2}\beta_i) + (\sum_{j\in C_1} \gamma_j + \sum_{j\in A_2}\alpha_j) + (\sum_{k\in A_1} \beta_k + \sum_{k\in B_2}\gamma_k)\\
\leq &\ \sum_{i\in A} \mu_i + \sum_{j\in B} \nu_j + \sum_{k\in C} \lambda_k. \end{align*}
\qed

\begin{Remark} 
Using the above argument, one can show that other inequalities of the form (\ref{Ineq:Grand}) hold. For example,
we can replace (III) by (III)(3)' and replace (II) by
\begin{itemize}
\item[(II)'] $|A| \geq  \max(|B'|,|C'|),|B| \geq \max(|A'|,|C'|), |C| \geq  \max(|A'|,|B'|)$.
\end{itemize}
Any $(\mu,\nu,\lambda)\in {\sf Par}_n^3$ such that $N_{\mu,\nu,\lambda}>0$ satisfies the corresponding inequality.
\end{Remark}

\subsection{Special subclasses of the inequalities} In \cite{GOY} we proved:
\begin{theorem}[Extended Weyl inequalities]
\label{thm:Horn}
Let $(\mu,\nu,\lambda)\in {\sf Par}_n^3$ and  
$1\leq k \leq i < j \leq l \leq n$, let $m = \min(i-k,l-j)$ and $M = \max(i-k,l-j)$.  If $N_{\mu,\nu,\lambda}>0$ then
\begin{equation}
\label{ijklInequality}
\mu_i - \mu_j \leq \lambda_k - \lambda_l + \nu_{m-p+1} + \nu_{M+p+2}, \text{\ \ \ where $0\leq p\leq m$}.
\end{equation}
\end{theorem}

\begin{definition}
For disjoint $X,Y\subseteq [n]$, the subset-sum inequalities are
\begin{equation}
\label{sumsubset}
0\leq \sum_{i\in X} \mu_i + \sum_{i\in Y} \nu_i - \sum_{i\in Y} \mu_i - \sum_{i\in X} \nu_i + \sum_{i=1}^{|X|+|Y|}\lambda_i.
\end{equation}
\end{definition}

\begin{proposition}
\label{prop:specificineq}
Any $(\mu,\nu,\lambda)\in {\sf Par}_n^3$ that satisfy all of the extended Horn inequalities \eqref{Ineq:Grand} also satisfy all $\mathfrak{S}_3$-permutations (\ref{eqn:S3symmetry}) of:
\begin{itemize}
\item[(i)]
the Horn inequalities \eqref{eq:ineq};
\item[(ii)]
the extended Weyl inequalities \eqref{ijklInequality};
\item[(iii)] the subset-sum inequalities \eqref{sumsubset};
\item[(iv)]
the triangle inequalities on $|\mu|$, $|\nu|$, and $|\lambda|$;
\item[(v)]
$|\mu \wedge \nu| \geq \frac{|\mu| + |\nu| - |\lambda|}{2}$.
\end{itemize}
\end{proposition}
\begin{proof}
A Horn inequality is of the form 
\[\sum_{i\in X} \mu_i \leq \sum_{j\in Y} \nu_j + \sum_{k\in Z} \lambda_k\] 
for $X,Y,Z\subseteq [n]$ and $c^{\tau(X)}_{\tau(Y),\tau(Z)} >0$.  Letting 
\[(A,A',B,B',C,C'):=(\emptyset,X,Y,\emptyset,Z,\emptyset)\] 
with $(A_1,A_2,B_1,B_2,C_1,C_2)$ as $(Z,Y,\emptyset,\emptyset,\emptyset,\emptyset)$ shows this is an extended Horn inequality. The Littlewood-Richardson positivity conditions (III)(2) and (III)(3) clearly hold.
Any ${\mathfrak S}_3$-symmetry of the Horn inequality is of the form (\ref{Ineq:Grand}), since
being of the form (\ref{Ineq:Grand}) is evidently preserved under ${\mathfrak S}_3$.

Similarly, let 
\[(A,A',B,B',C,C'):=(\{j\},\{i\},\{m-p+1,M+p+2\},\emptyset,\{k\},\{l\})\] 
with $(A_1,A_2,B_1,B_2,C_1,C_2)$ as $(\{k\},\{i-k+1\},\emptyset,\emptyset,\{l-j+1\},\{j\})$. Let us only comment on
the assertion $c^{\tau(B)}_{\tau(C_1),\tau(A_2)}=c^{(M+p,m-p)}_{(l-j),(i-k)}>0$, which is true by Pieri's rule. Thus,
\eqref{ijklInequality} is of the form (\ref{Ineq:Grand}).
 
The subset-sum inequalities are of type \eqref{Ineq:Grand}. Let
\[(A,A',B,B',C,C'):=(X,Y,Y,X,[|X|+|Y|],\emptyset)\]
with $(A_1,A_2,B_1,B_2,C_1,C_2):=([|Y|],Y,X,[|X|],\emptyset,\emptyset)$.  By letting $X=[n]$ and $Y=\emptyset$, the triangle inequalities are cases of the subset-sum inequalities. The verification is clear.

Let $X := \{i\in [n]: \mu_i\leq \nu_i\}$, and let $Y := \{i\in [n]: \mu_i> \nu_i\} = X^c$.  Then \[|\mu\wedge \nu| = \sum_{i\in X} \mu_i + \sum_{j\in Y} \nu_j,\] and so (v) can be rewritten as
\begin{equation*}
2\sum_{i\in X} \mu_i + 2\sum_{j\in Y} \nu_j \geq \sum_{i=1}^n \mu_i +  \sum_{j=1}^n \nu_j - \sum_{k=1}^n \lambda_k
\end{equation*}
or
\begin{equation*}
0 \leq \sum_{i\in X} \mu_i -  \sum_{i\in Y} \mu_i +  \sum_{j\in Y} \nu_j - \sum_{j\in X} \nu_j + \sum_{k=1}^n \lambda_k
\end{equation*}

This is a subset-sum inequality, and we are done by (iii).
\end{proof}

\begin{example}
The extended Horn inequalities for $n=2$ are the $\mathfrak{S}_3$-permutations (\ref{eqn:S3symmetry}) of:
\begin{equation}
\label{eqn:n=2a}
0\leq \mu_1 + \nu_1 - \lambda_1
\end{equation}
\begin{equation}
\label{eqn:n=2b}
0\leq \mu_1 + \nu_2 -\lambda_2
\end{equation}
\begin{equation}
\label{eqn:n=2c}
0\leq \mu_1 + \mu_2 + \nu_1 + \nu_2 -\lambda_1 -\lambda_2
\end{equation}
\begin{equation}
\label{eqn:n=2d}
0\leq \mu_1 -\mu_2 -\nu_1 + \nu_2 + \lambda_1 + \lambda_2
\end{equation}
where (\ref{eqn:n=2a}) and (\ref{eqn:n=2b}) are Horn inequalities, (\ref{eqn:n=2c}) is a triangle inequality, and (\ref{eqn:n=2d}) is both an extended Weyl inequality and a subset-sum inequality.
\end{example}

\begin{example}
The extended Horn inequalities for $n=3$ are the $\mathfrak{S}_3$-permutations of the $n=3$ Horn inequalities, extended Weyl inequalities, subset-sum inequalities,
\begin{equation}
\label{eqn:redundant3a}
0\leq -\mu_1 +\mu_2 +\mu_3 +\nu_1 - \nu_2 + \nu_3+ \lambda_1 + \lambda_2-\lambda_3, 
\end{equation}
\begin{equation}
\label{eqn:redundant3b}
0\leq \mu_1 -\mu_2 +\mu_3 +\nu_1 - \nu_2 + \nu_3+ \lambda_1 - \lambda_2+\lambda_3, 
\end{equation}
\begin{equation}
\label{eqn:redundant3c}
0\leq \mu_1 -\mu_2 +\nu_1 - \nu_2  + \lambda_2+\lambda_3, 
\end{equation}
\begin{equation}
\label{eqn:redundant3d}
0\leq \mu_2 -\mu_3 +\nu_2 - \nu_3 + \lambda_2+\lambda_3.
\end{equation}
\end{example}

\begin{example}[Minimal inequalities?]
The Horn inequalities \eqref{eq:ineq} are redundant. One can shorten the list to those where 
\[c_{\tau(I),
\tau(J)}^{\tau(K)}=1;\] 
this is a result of P.~Belkale \cite[Theorem~9]{Belkale} (conjectured by C.~Woodward). That each of these inequalities are 
essential is proved by A.~Knutson-T.~Tao-C.~Woodward \cite[Section~6]{KTW}. We know of no na\"ive analogue
of these results. Specifically, the inequalities (\ref{eqn:redundant3b}), (\ref{eqn:redundant3c}), and (\ref{eqn:redundant3d}) 
are redundant, since they are implied by the ``partition inequalities'', \emph{i.e.}, $\mu,\nu,\lambda\in {\sf Par}_3$.  However, all Littlewood-Richardson coefficients associated with (\ref{eqn:redundant3c}) and (\ref{eqn:redundant3d}) are $1$. 
 \end{example}

 \begin{proposition}
Conjecture~\ref{thm:Classification} holds for $n=2$.
\end{proposition}
\begin{proof}
We prove the contrapositive. Suppose $N_{\mu,\nu,\lambda}=0$ and (\ref{eqn:parity}) holds. By 
\cite[Theorem~5.14]{GOY} either an $n=2$ Horn inequality or extended Weyl inequality fails. By Proposition~\ref{prop:specificineq}(ii), the inequalities \eqref{Ineq:Grand} include
the Horn inequalities and the extended Weyl inequalities. Therefore, an inequality  \eqref{Ineq:Grand} is violated.
\end{proof}

\section{Proof of Theorem~\ref{thm:rowcol}}\label{sec:3}

First consider the case where one of the partitions is a single row.  Without loss of generality, $\lambda = (p)$.  By Proposition~\ref{prop:specificineq}(i), $(\mu,\nu,\lambda)$ satisfies all $\mathfrak{S}_3$ permutations of the Horn inequalities.  In particular, they satisfy $\mu_{i+1} \leq \nu_i + \lambda_2$ and $\nu_{i+1} \leq \mu_i + \lambda_2$ for all $i\in [n-1]$.  This implies that $\mu_{i+1} \leq \nu_i $ and $\nu_{i+1} \leq \mu_i$, so $\mu_{i+1},\nu_{i+1}\leq (\mu \wedge \nu)_{i}$ for all $i\in [n-1]$, which is equivalent to saying that $\mu/(\mu \wedge \nu)$ and $\nu/(\mu \wedge \nu)$ are horizontal strips.

Let $k:=\frac{|\mu|+|\nu|-p}{2}$.  Since Proposition~\ref{prop:specificineq}(iv) says that $(\mu,\nu,\lambda)$ satisfies the triangle inequalities, $k\geq 0$. Moreover, $|\mu|+|\nu|+|\lambda|$ is even (by hypothesis), hence $k\in \mathbb{Z}_{\geq 0}$.  Proposition \ref{prop:specificineq}(v) also says that $k\leq |\mu\wedge \nu|$.

\begin{claim}
There exist at least $|\mu\wedge \nu|-k$ columns $i$ such that $\mu'_i = \nu'_i >0$.
\end{claim}
\begin{proof}
Without loss of generality, say that $\mu_1\leq \nu_1$.  Since $\mu/(\mu \wedge \nu)$ is a horizontal strip, there are $|\mu/(\mu \wedge \nu)| = |\mu|-|\mu \wedge \nu|$ columns $i$ such that $\mu'_i > \nu'_i$, and similarly there are $|\nu|-|\mu \wedge \nu|$ columns $i$ such that $\mu'_i < \nu'_i$.  Since there are $\nu_1$ columns where at least one of $\mu'_i$ or $\nu'_i$ is nonzero, this means that there are $\nu_1 - |\mu| - |\nu| + 2|\mu \wedge \nu|$ columns such that $\mu'_i = \nu'_i >0$, so it suffices to prove that \[\nu_1 - |\mu| - |\nu| + 2|\mu \wedge \nu|\geq |\mu\wedge \nu|-k.\]
Rearranging the terms and substituting in for the definition of $k$, this becomes 
\begin{equation}
\label{eqn:previousSept3}
0 \leq \nu_1 - |\mu| - |\nu| + |\mu \wedge \nu| + \frac{|\mu|+|\nu|-p}{2}.
\end{equation}

Define 
\[X := \{i\in [n]:\mu_i> \nu_i\} \text{\ and $Y := \{i\in [n]:\mu_i\leq \nu_i\} = X^c$.}\] 
By assumption, $1\in Y$.  Hence from \eqref{eqn:previousSept3} and the definition of $|\cdot |$, 
\begin{align*}
0&\leq 2\nu_1 - 2|\mu| - 2|\nu| + 2|\mu \wedge \nu| + (|\mu|+|\nu|-p)\\
& = 2\nu_1 - |\mu| - |\nu| + 2|\mu \wedge \nu| -p
\end{align*}
\begin{align*}
& = 2\nu_1 - \sum_{i=1}^n \mu_i - \sum_{i=1}^n \nu_i + 2\sum_{i\in X} \nu_i + 2\sum_{i\in Y} \mu_i -p\\
& = 2\nu_1 - \sum_{i\in X} \mu_i - \sum_{i\in Y} \nu_i + \sum_{i\in X} \nu_i + \sum_{i\in Y} \mu_i -p\\
& =  - \sum_{i\in X} \mu_i - \sum_{i\in Y\setminus \{1\}} \nu_i + \sum_{i\in X\cup \{1\}} \nu_i + \sum_{i\in Y} \mu_i -p.
\end{align*}
However, this is always true, since 
\[(Y,X,X\cup\{1\},Y\setminus\{1\},[n]\setminus\{1\},\{1\})\in \mathcal{G}_n,\] 
which can be seen by letting 
\begin{multline}\nonumber
(A_1,A_2,B_1,B_2,C_1,C_2) = \\([|X|+1]\setminus\{1\},\{i-1:i\in X\},\{i-1\in \mathbb{Z}_{>0}: i\in Y\},[|Y|]\setminus\{1\},\{1\},\{1\}).
\end{multline}
The verification of $c_{\tau(C_1),\tau(A_2)}^{\tau(B)}>0$ relies on 
\[\tau(\{i-1:i\in X\})=\tau(X\cup \{1\}).\] 
Similarly one checks
$c^{\tau(A)}_{\tau(B_1),\tau(C_2)}>0$.
\end{proof}

Let $\alpha$ be the partition formed by removing the southernmost box from the $|\mu\wedge \nu|-k$ rightmost columns $i$ of $\mu\wedge \nu$ such that $\mu'_i = \nu'_i >0$.  Since $\mu\wedge\nu=(\mu'\wedge\nu')'$, the boxes removed from $\mu\wedge\nu$ to form $\alpha$ are in different columns than the boxes removed from $\mu$ to form $\mu\wedge\nu$ or from $\nu$ to form $\mu\wedge\nu$. Thus, $\mu/\alpha$ and $\nu/\alpha$ are both horizontal strips.  Also,
\begin{align*}
|\mu/\alpha| &= |\mu/(\mu\wedge\nu)| + |(\mu\wedge\nu)/\alpha|\\
& = (|\mu| -|\mu\wedge\nu|) + (|\mu\wedge\nu|-k) \\
& = |\mu|  - \frac{|\mu|+|\nu|-p}{2}\\
& = \frac{|\mu| + p - |\nu|}{2}
\end{align*}
and similarly $|\nu/\alpha| = \frac{|\nu| + p - |\mu|}{2}$.

As a result, one can remove a horizontal strip from $\mu$ of length $ \frac{|\mu| + p - |\nu|}{2}$, and then add a horizontal strip of length $\frac{|\nu| + p - |\mu|}{2}$ back in to result in $\nu$.  This is exactly the statement of Proposition 2.4 of \cite{GOY}, so $N_{\mu,\nu, (p)}>0$.

Now consider the case where one of the partitions is a single column.  Without loss of generality, $\lambda = (1^p)$.  By Proposition \ref{prop:specificineq}, $(\mu,\nu,\lambda)$ satisfies all $\mathfrak{S}_3$ permutations of the Horn inequalities.  In particular, they satisfy $\mu_{i} \leq \nu_i + \lambda_1$ and $\nu_{i} \leq \mu_i + \lambda_1$ for all $i\in [n]$.  This implies that $\mu_{i} \leq \nu_i +1$ and $\nu_{i} \leq \mu_i + 1$, so $\mu_i,\nu_i\leq (\mu \wedge \nu)_{i}+1 $ for all $i\in [n+1]$, which is equivalent to saying that $\mu/(\mu \wedge \nu)$ and $\nu/(\mu \wedge \nu)$ are vertical strips.

Let $k:=\frac{|\mu|+|\nu|-p}{2}$.  As before, $k\in \mathbb{Z}_{\geq 0}$ and $k\leq |\mu\wedge \nu|$.
\begin{claim}
There exist at least $|\mu\wedge \nu|-k$ rows $i$ such that $\mu_i = \nu_i >0$.
\end{claim}
\begin{proof}
Since $\mu/(\mu \wedge \nu)$ is a vertical strip, there are $|\mu/(\mu \wedge \nu)| = |\mu|-|\mu \wedge \nu|$ rows $i$ such that $\mu_i > \nu_i$, and similarly there are $|\nu|-|\mu \wedge \nu|$ rows $i$ such that $\mu_i < \nu_i$.  Let $L = \max(\ell(\mu),\ell(\nu))$.  Since there are $L$ rows where at least one of $\mu_i$ or $\nu_i$ is nonzero, this means that there are $L - |\mu| - |\nu| + 2|\mu \wedge \nu|$ rows such that $\mu_i = \nu_i >0$, so it suffices to prove that \[L - |\mu| - |\nu| + 2|\mu \wedge \nu|\geq |\mu\wedge \nu|-k.\]
Rearranging the terms and substituting in for the definition of $k$, this becomes \[0 \leq L - |\mu| - |\nu| + |\mu \wedge \nu| + \frac{|\mu|+|\nu|-p}{2}.\]

Define 
\[X := \{i\in [L]:\mu_i> \nu_i\} \text{\ and $Y := \{i\in [L]:\mu_i\leq \nu_i\} = [L]\setminus X$.}\]  
Multiplying the above expression by $2$ and using the definition of $|\cdot |$, we get:
\begin{align*}
0&\leq 2L - 2|\mu| - 2|\nu| + 2|\mu \wedge \nu| + (|\mu|+|\nu|-p)\\
& = 2L - |\mu| - |\nu| + 2|\mu \wedge \nu| -p\\
& = 2L - \sum_{i=1}^L \mu_i - \sum_{i=1}^L \nu_i + 2\sum_{i\in X} \nu_i + 2\sum_{i\in Y} \mu_i -p\\
& = 2L - \sum_{i\in X} \mu_i - \sum_{i\in Y} \nu_i + \sum_{i\in X} \nu_i + \sum_{i\in Y} \mu_i -p
\end{align*}

We split the remainder of the proof of the claim into two cases: whether $L<p$ or $L\geq p$.

\noindent \emph{Case 1: ($L< p$)}  Here we can rewrite the above inequality as \[0\leq  - \sum_{i\in X} \mu_i - \sum_{i\in Y} \nu_i + \sum_{i\in X\cup ([p]\setminus [L])} \nu_i + \sum_{i\in Y\cup ([p]\setminus [L])} \mu_i + L-(p-L)\]

However, this is always true, since
\[(Y\cup ([p]\setminus [L]),X,X\cup ([p]\setminus [L]),Y,[L],[p]\setminus [L])\in \mathcal{G}_n\]
which can be seen by letting 
\begin{multline}\nonumber
(A_1,A_2,B_1,B_2,C_1,C_2) = ([|X|],X,Y,[|Y|],[|Y|+p-L]\setminus [|Y|],[|X|+p-L]\setminus [|X|]).
\end{multline}
The slightly trickier verification needed from (III)(3) is 
\[c_{\tau(C_1),\tau(A_2)}^{\tau(B)} =c_{\tau([|Y|-p+L]-[|Y|]),\tau(X)}^{\tau(X\cup ([p]-[L]))}
 =c_{|Y|^{p-L},\tau(X)}^{\tau(X)+(L-|X|)^{p-L}}
 =c_{|Y|^{p-L},\tau(X)}^{\tau(X)+|Y|^{p-L}}>0;\]
the latter is obvious.  The check $c^{\tau(A)}_{\tau(B_1),\tau(C_2)}>0$ is analogous.

\noindent \emph{Case 2: ($L\geq p$)}  Here we can instead rewrite the above inequality as \[0\leq  2(L-p) - \sum_{i\in X} \mu_i - \sum_{i\in Y} \nu_i + \sum_{i\in X} \nu_i + \sum_{i\in Y} \mu_i + p,\]
and so it suffices to show 
\[0\leq  - \sum_{i\in X} \mu_i - \sum_{i\in Y} \nu_i + \sum_{i\in X} \nu_i + \sum_{i\in Y} \mu_i + p.\]

This is true  since
$(Y,X,X,Y,[L],\emptyset)\in \mathcal{G}_n$
is just a subset-sum inequality.
\end{proof}

Let $\alpha$ be the partition formed by removing the rightmost box from the $|\mu\wedge \nu|-k$ southernmost rows $i$ of $\mu\wedge \nu$ such that $\mu_i = \nu_i >0$.  Since the boxes removed from $\mu\wedge\nu$ to form $\alpha$ are in different rows than the boxes removed from $\mu$ to form $\mu\wedge\nu$ or from $\nu$ to form $\mu\wedge\nu$, $\mu/\alpha$ and $\nu/\alpha$ are both vertical strips.  In addition,
\begin{align*}
|\mu/\alpha| &= |\mu/(\mu\wedge\nu)| + |(\mu\wedge\nu)/\alpha|\\
& = (|\mu| -|\mu\wedge\nu|) + (|\mu\wedge\nu|-k) \\
& = |\mu|  - \frac{|\mu|+|\nu|-p}{2}\\
& = \frac{|\mu| + p - |\nu|}{2}
\end{align*}
and similarly $|\nu/\alpha| = \frac{|\nu| + p - |\mu|}{2}$.

As a result, one can remove a vertical strip from $\mu$ of length $ \frac{|\mu| + p - |\nu|}{2}$, and then add a vertical strip of length $\frac{|\nu| + p - |\mu|}{2}$ back in to result in $\nu$.  This is exactly the conjugate statement of Proposition 2.4 of \cite{GOY}, so $N_{\mu,\nu, (1^p)}>0$.\qed

\section*{Acknowledgements}
We thank Allen Knutson and Nicolas Ressayre for helpful remarks on an earlier preprint version, respectively, concerning the history of Theorem~\ref{thm:classicalHorn}, and the relationship between \cite{BK} and this work. We are grateful to David Brewster for writing code essential for checking Conjecture~\ref{thm:Classification}, as part
of his involvement in the NSF RTG funded ICLUE summer program. We also thank the anonymous referees for
their useful comments. We used Anders Buch's Littlewood-Richardson calculator in our experiments. This work was partially supported by a Simons Collaboration grant, 
funding from UIUC's Campus Research Board, and an NSF RTG grant. SG was partially supported by the
National Science Foundation Graduate Research Fellowship Program under Grant No. DGE--1746047.

AY would like to add an additional acknowledgement to Ian Goulden and David Jackson, who were instrumental
in said authors' professional development at Waterloo. Your words were heard, and made a difference.


\begin{thebibliography}{999}
\bibitem{Belkale} P.~Belkale, \emph{Local systems on $\Bbb P^1-S$ for $S$ a finite set.} 
Compositio Math. 129 (2001), no. 1, 67--86.

\bibitem{BK} P.~Belkale and S.~Kumar, \emph{Eigenvalue problem and a new product in cohomology of flag varieties.} Invent. Math. 166 (2006), no. 1, 185--228.

\bibitem{Bhatia} R.~Bhatia, \emph{Linear algebra to quantum cohomology: the story of Alfred Horn's 
inequalities.} Amer. Math. Monthly 108 (2001), no. 4, 289--318.

\bibitem{BOR} E.~Briand, R.~Orellana, and M.~Rosas, \emph{Rectangular symmetries for coefficients of symmetric functions.} Electron. J. Combin. 22 (2015), no.3, Paper 3.15, 18.

\bibitem{Fulton} W.~Fulton, \emph{Eigenvalues, invariant factors, highest weights, and Schubert calculus.} Bull. Amer. Math. Soc. (N.S.) 37 (2000), no. 3, 209--249. 

\bibitem{GOY} S.~Gao, G.~Orelowitz, and A.~Yong, \emph{Newell-Littlewood numbers}, preprint, 2020.
\textsf{arXiv:2005.09012}.  
\bibitem{NLIII} S.~Gao, G.~Orelowitz, N.~Ressayre, and A.~Yong, \emph{Newell-Littlewood numbers III:
eigencones and GIT-semigroups}, preprint, 2021. \textsf{arXiv:2107.03152}

\bibitem{Hahn} H.~Hahn, \emph{On classical groups detected by the triple tensor product and the Littlewood-Richardson semigroup.} Res. Number Theory 2 (2016), Paper No. 19, 12 pp. 
\bibitem{Horn} A.~Horn, \emph{Eigenvalues of sums of Hermitian matrices.} Pacific J. Math. 12 (1962), 225--241.
\bibitem{Klyachko} A.~Klyachko, \emph{Stable bundles, representation theory and Hermitian operators.} Selecta Math. (N.S.) 4 (1998), no. 3, 419--445.
\bibitem{KT} A.~Knutson and T.~Tao, \emph{The honeycomb model of ${\rm GL}_n({\Bbb C})$ tensor products. I. Proof of the saturation conjecture.} J. Amer. Math. Soc. 12 (1999), no. 4, 1055--1090.
\bibitem{KTW} A.~Knutson, T.~Tao, and C.~Woodward, \emph{The honeycomb model of ${\rm GL}_n(\Bbb C)$ tensor products. II. Puzzles determine facets of the Littlewood-Richardson cone.} J. Amer. Math. Soc. 17 (2004), no. 1, 19--48.
\bibitem{KM} S.~Kumar, \emph{A survey of the additive eigenvalue problem.} With an appendix by M. Kapovich. Transform. Groups 19 (2014), no. 4, 1051--1148.
\bibitem{Littlewood} D.~E.~Littlewood, \emph{Products and plethysms of characters with orthogonal, symplectic and symmetric groups.} Canadian J. Math. 10 (1958), 17--32.
\bibitem{Newell} M.~J.~Newell, \emph{Modification rules for the orthogonal and symplectic groups.} Proc. Roy. Irish Acad. Sect. A 54, (1951). 153--163.
\bibitem{Ress} N.~Ressayre, \emph{Horn inequalities for nonzero Kronecker coefficients.} Adv. Math. 356 (2019), 106809, 21 pp. 
\bibitem{Zel} A.~Zelevinsky, \emph{Littlewood-Richardson semigroups.} New perspectives in algebraic combinatorics (Berkeley, CA, 1996--97), 337--345, Math. Sci. Res. Inst. Publ., 38, Cambridge Univ. Press, Cambridge, 1999. 
\end{thebibliography}
\end{document}